\documentclass[12pt]{article} 

\usepackage{amsmath, amsfonts, amssymb}
\usepackage{mathrsfs}
\usepackage{theorem}
\usepackage{bm}

\RequirePackage[colorlinks,citecolor=blue,urlcolor=blue,bookmarksopen]{hyperref}

\newtheorem{mythm}{Theorem}[section]

\newtheorem{mylem}[mythm]{Lemma}

\newtheorem{mydefn}[mythm]{Definition}

{\newtheorem{myrem}[mythm]{Remark}}
{}

\newcommand{\ra}{\rightarrow}
\newcommand{\dis}{\displaystyle}

\def\R{\mathbb R}

\def\N{\mathbb N}
\def\C{\mathscr C}

\def\F{\mathscr F}
\def\d{\text{\rm{d}}}
\def\E{\mathbb E}
\def\p{\mathbb P}
\def\e{\text{\rm{e}}}

\def\La{\Lambda}
\def\veps{\varepsilon}

\def\ll{\mathscr{L}}

\newcommand{\D}{\mathscr D}
\def\S{\mathcal S}
\def\C{\mathscr C}
\def\vsig{\varsigma}
\def\pb{\mathscr{P}}

\def\wt{\widetilde}

\newcommand{\fin}{\hfill $\square$\par}
\newenvironment{proof}{{\noindent\it Proof.}\ }{\hfill $\square$\par}

\numberwithin{equation}{section}

\allowdisplaybreaks

\begin{document}

\title{The existence of optimal feedback controls for stochastic dynamical systems with regime-switching\thanks{Supported in
 part by NNSFs of China (Nos. 11771327,  11831014)}}

\author{Jinghai Shao\thanks{Center for Applied Mathematics, Tianjin University, Tianjin 300072, China. (Email: shaojh@tju.edu.cn.) } }

\maketitle

\begin{abstract}
In this work we provide explicit conditions on the existence of optimal feedback controls for stochastic processes with regime-switching. We use the compactification method which needs less regularity conditions on the coefficients of the studied stochastic systems. Two different kinds of controls have been considered: one is the control on the coefficients of the diffusion processes, another is the control on the transition rate matrices of the continuous-time Markov chains. Moreover, the dynamic programming principle is established after showing the continuity of the value function.
\end{abstract}



\section{Introduction}

This work focuses on providing sufficient conditions for the existence of optimal feedback controls for the stochastic control systems with regime-switching. This system contains two components $(X_t,\La_t)$: the continuous component $(X_t)$ satisfies a stochastic differential equation (SDE) which describes the evolution of the studied dynamical system; the discrete component $(\La_t)$ is a jumping process on a finite state space which describes the random change of the environment in which $(X_t)$ lives. The control policy also owns two terms: one is to control the coefficients of SDEs; another is to control the transition rate matrices of $(\La_t)$. This kind of controls is of great meaning in applications and has not been investigated before.  All admissible control policies considered in this paper are in the form of feedback control. We develop the compactification method to provide explicit conditions to guarantee the existence of optimal feedback controls with respect to finite-horizon cost functions.  The value function is shown to be continuous and  the dynamic programming principle is established.

%

The existence of optimal feedback controls is a fundamental issue in the study of control theory. This issue is not only theoretical, since it is needed to ensure that the optimization problem is well defined and to allow subsequent analysis of the equations for the value function.

One approach to establish the existence of optimal controls is based on the theory of partial differential equations of dynamic programming; see the early works of Davis \cite{Dav} and Bismut \cite{Bis}, Fleming and Rishel \cite{FR} or the recent survey 
Kushner \cite{Ku14} and the references therein.  This method has been extensively studied in connection with the theory of Hamilton-Jacobi-Bellman equations, which encounters the restriction of the regularity of corresponding solutions. Another approach is to show directly the compactness of the minimizing sequence of controls. Kushner \cite{Ku} used the weak convergence of measures to provide a general result on the existence of optimal controls.  Also, Haussmann and Lepeltier \cite{HL90}, Haussmann and Suo
\cite{HS95, HS95b} have developed this method to show the existence of optimal controls and even optimal relaxed controls. The advantage of this compactification method is that it requires less regularity of the value function and thus needs only very mild hypothesis on the data. Especially, the works \cite{Ku, HL90, HS95, HS95b} investigated the stochastic open loop problem.  Moreover, given the existence of an optimal control,  \cite{HL90} used Krylov's Markov selection theorem showed the optimal control could be represented as a Markov control.
It is a far more trivial task to guarantee the limit of the minimizing sequence being adapted to the stochastic fields generated by the dynamic system which  also strongly depends on the limit of the control sequence. In view of this difficulty, the known sufficient conditions on the existence of optimal feedback controls for stochastic control models are mostly provided by the theory of partial differential equations; see, for example, Fleming and Rishel \cite[Chapter VI]{FR} and references therein. A little more precisely, under the help of the verification theorem, once the existence of appropriate solution to certain nonlinear parabolic equation is known, the existence of an optimal control follows from a measurable selection theorem, and the corresponding controlled system is then defined with the aid of the Girsanov formula. Following this approach, under the condition that the diffusion coefficients are independent of the controls, applying the theory of nonlinear parabolic equations, Fleming and Rishel \cite[Theorem 6.3]{FR} presented a result on the existence of the optimal feedback control. However, this method meets an essential difficulty to deal with the system with control-dependent diffusion coefficients, since the optimal feedback controls may be discontinuous.

Besides, Linquist \cite{Lin} transformed the feedback control problem into the stochastic open loop problem for a class of linear systems by adding a further restriction on his feedback class. In this work, we shall develop the compactification method to provide sufficient conditions on the existence of optimal feedback controls.

The stochastic maximum principle plays a central role in stochastic control theory. It gives necessary conditions for optimal controls. Its first version was established by Kushner \cite{Ku72} where the diffusion coefficients are independent of the controls, and  by Peng \cite{Peng} when the diffusion coefficients depend on the controls. Some advance information about the form of the optimal control is needed to use stochastic maximum principles to find optimal controls in applications. For example, L\"u, Wang, and Zhang \cite{LWZ} established the equivalence between the existence of optimal feedback controls for the stochastic linear quadratic control problems and the solvability of the corresponding backward stochastic Riccati equations in some sense. Recently, H. Zhang and X. Zhang \cite{ZZ} investigated the second-order necessary conditions of the optimal controls.

Another purpose of this work is to study the optimal control problem on the transition rates for the dynamic systems living in a random environment, which is characterized via a continuous-time Markov chain, known also as regime-switching processes.
Variations in the external environment (for example, weather or temperature) can have important effects on the dynamics of the studied systems. For instance, for the ecosystem, certain biological parameters such as the growth rates and the carrying capacities often demonstrate abrupt changes due to the environmental noise. Therefore, it is natural to consider the random changes of the environment in mathematical modeling. Recently, such models are widely applied in stochastic control and optimization, mathematical finance, ecological and biological systems, engineer, etc.; see, for example, \cite{BS16, DDY, HS17, M13, YZ} amongst others.
In view of its wide application, this optimal control problems for regime-switching processes have been studied in the literature; see, for instance,
\cite{Gho, SSZ, SZ16, ZSM}, and \cite{ZY03} amongst others. In particular,  \cite{SSZ} and \cite{SZ16} investigated the singular control problem for regime-switching processes with Markovian regime-switching and pointed out the difference in the optimal control problem for the dynamic systems with and without switching. In \cite{SSZ}, Song et al. showed that the value function is a viscosity solution of a system of quasi-variational inequalities (QVIs) through proving first the continuity of the value function by exploiting the advantage of a one-dimensional regime-switching diffusion process. For the Markovian regime-switching processes in high dimensional space, by establishing directly a weakly dynamic programming principle instead of proving the continuity of the value function, Song and Zhu in \cite{SZ16} showed directly that the value function is a viscosity solution to a system of QVIs. However,
there is no discussion on the optimal control problem for state-dependent regime-switching processes which is more complicated due to the intensive interaction between the state process and the switching process. In addition, the result on the control of the transition rate matrix is very limited up to our knowledge.

The regime-switching diffusion processes $(X_t,\La_t)_{t\geq 0}$ contains two components: the first component $(X_t)_{t\geq 0}$ satisfies the following SDE:
\begin{equation}\label{1.1}
\d X_t=b(X_t,\La_t)\d t+\sigma(X_t,\La_t)\d B_t,
\end{equation} where $b:\R^d\times \S\ra \R^d$, $\sigma:\R^d\times \S\ra \R^{d\times d}$, and $(B_t)_{t\geq 0}$ is a standard $d$-dimensional Brownian motion;
the second component $(\La_t)_{t\geq 0}$ is continuous-time jumping process satisfying
\begin{equation}\label{1.2}
\p(\La_{t+\delta}=j|\La_t=i,\ X_t=x)=\begin{cases}
                                        q_{ij}(x)\delta+o(\delta), & \mbox{if $j\neq i$ },  \\
                                        1+q_{ii}(x)\delta+o(\delta), & \mbox{otherwise},
                                      \end{cases}
\end{equation} provided $\delta>0$. The component $(X_t)_{t\geq 0}$ is used to describe the evolution of a dynamical system, and the component $(\La_t)_{t\geq 0}$ is used to reflect the random switching of the environment where the studied system lives. When the transition rate matrix $(q_{ij}(x))$ depends on $x$, $(X_t,\La_t)_{t\geq 0}$ is called a state-dependent regime-switching process. When $(q_{ij}(x))$ does not depend on $x$, then $(\La_t)_{t\geq 0}$ is indeed a continuous-time Markov chain, and is assumed to be independent of the Brownian motion $(B_t)_{t\geq 0}$ as usual. In this case, $(X_t,\La_t)_{t\geq 0}$ is a state-independent regime-switching process, and sometimes called Markovian regime-switching process. Here $\S$ is a denumerable space, and $\mathscr B(\S)$ denotes the collection of all measurable sets. When $\S$ is a finite set, various properties of regime-switching processes such as stability, ergodicity, numerical approximation, etc. have been widely studied in the literature; see, e.g. \cite{M13,MY,YZ,Sh18,XZ,YZ} and references therein. When $\S$ is an infinitely countable set, we refer to \cite{SX,Sh15a,Sh15b}, where two kinds of methods, finite partition method and principle eigenvalue method, were raised to deal with the stability and ergodicity of regime-switching processes.

In this work we use the compactification method to show the existence of the optimal feedback control with respect to a very general finite-horizon cost function. Here our method looks similar to Haussmann and Suo \cite{HS95}, but the technics are quite different. This can be easily seen from the fact that \cite{HS95} cannot deal with the case that the cost function depends on the terminal value the process, but we can. Similar to \cite{HS95}, the feedback controls considered in this paper is a probability measure valued process, which is a kind of relaxed control. The ordinary control policies, i.e. controls taking values in some subset $U$ of the Euclidean space, can be viewed as a special kind of relaxed controls via identifying the point $x\in U$ with the Dirac measure $\delta_x$. See, for example, \cite{HL90} for the existence of optimal ordinary control; \cite[Theorem 3.6]{HS95} for some discussions on deriving the existence of optimal ordinary control from the existence of optimal relaxed control. Moreover, the dynamic programming principle is established in the end of this work, which enables us to study further the (viscosity) solution to the corresponding Hamilton-Jacobi-Bellman equation in the future.

This paper is organized as follows. In Section 2, we introduce the class of admissible feedback controls and prove the existence of the optimal feedback control by using compactification method. Section 3 is devoted to  establishing the dynamic programming principle.

\section{Existence of optimal controls}

\subsection{Framework and statement of the result}
Let $\S=\{1,2,\ldots,N\}$ with $N<\infty$. $T$ is a positive constant given throughout this work. $U$ is a compact set of, say, $\R^k$ for some $k\in \N$, and $\pb(U)$ denotes the collection of all probability measures over $U$. For any two probability measures $\mu$ and $\nu$ in $\pb(U)$,
their $L^1$-Wasserstein distance is defined as:
\[W_1(\mu,\nu)=\inf_{\Gamma\in \C(\mu,\nu)}\Big\{\int_{U\times U}\!\!|x-y|\Gamma(\d x,\d y)\Big\},\]
where $\C(\mu,\nu)$ stands for the set of all couplings of $\mu$ and $\nu$ on $U\times U$. See \cite[Chapter 7]{AGS} or \cite{Vi} for more discussions on the Wasserstein distance and the geometry of $\pb(U)$.

Let $E$ be a metric space. For $0\leq a<b\leq T$,
\begin{itemize}
\item $\mathcal{C}([a,b];\!E)$ is the collection of continuous functions $x:[a, b]\ra E$;
\item $\mathcal{D}([a,b];\!E)$ is the collection of right-continuous functions with left limits $x:\![a,b] \!\ra \!E$.
\end{itemize}

Denote by $x_{[s,t]}$ the function $x_\cdot$ in $\mathcal{C}([s,t];E)$ or $\mathcal{D}([s,t];E)$ with $s,t\in [0,T]$, and it can be extended to the whole interval $[0,T]$ through the map $\Xi$:
\begin{equation}\label{ext-1}
(\Xi x_{[s,t]})_r=\begin{cases}
  x_s,& \text{if $r\leq s$,}\\
  x_r,& \text{if $s< r <t$,}\\
  x_t,& \text{if $r\geq t$.}
\end{cases}
\end{equation}

Give a probability space $(\Omega,\F,\p)$ endowed with a complete filtration $\{\F_t\}_{t\geq 0}$. Consider the following stochastic dynamical system
\begin{equation}\label{l-1}
\d X_t=b(X_t,\La_t,\mu_t)\d t+\sigma(X_t,\La_t,\mu_t)\d B_t,
\end{equation}
where $b:\R^d\times\S\times \pb(U)\ra \R^d$, $\sigma:\R^d\times \S\times \pb(U)\ra \R^{d\times d}$, and $(B_t)_{t\geq 0}$ is a $d$-dimensional $\F_t$-Brownian motion. Here $(\La_t)_{t\geq 0}$ is a continuous-time jumping process on $\S$ satisfying
\begin{equation}\label{l-2}
\p(\La_{t+\delta}=j|\La_t=i,\ X_t=x,\nu_t=\nu)=\begin{cases}
                                        q_{ij}(x,\nu)\delta+o(\delta), & \mbox{if $j\neq i$ },  \\
                                        1+q_{ii}(x,\nu)\delta+o(\delta), & \mbox{otherwise},
                                      \end{cases}
\end{equation} provided $\delta>0$ for every $x\in \R^d$, $\nu\in \pb(U)$, $i,\,j\in\S$.
In this controlled system  \eqref{l-1} and \eqref{l-2}, we consider two kinds of controls: $\mu_\cdot$ and $\nu_\cdot$, which are both measurable maps from $[0,T]$ to $\pb(U)$. The term $\mu_\cdot$ is a kind of classical relaxed control for stochastic dynamical system which has been studied in many works. The  term $\nu_t$ is a special control policy for regime-switching processes, which is used to control the transition rate matrices of the jumping process $(\La_t)$. As $(\La_t)$ is a jumping process in a discrete state space, the role played by the control term $\nu_t$ is quite different to that played by the term $\mu_t$ in the evolution of the studied dynamic system. This kind of control $\nu_t$ has not been studied in the optimal control problem for regime-switching processes before. In addition,  this control is closely related to the control policy used in the study of continuous-time Markov decision processes (cf. e.g. \cite{Guo07, GH, GVZ} and references therein). See \cite{Sh19a} for more discussion on their relationship.

The feedback controls studied in this work are introduced as follows.

\begin{mydefn}\label{def-1}
For each  $(s,x,i)\!\in \! [0,T)\times \R^d\times \S$, a feedback control $\alpha=(\mu_t,\nu_t)_{t\in [s,T]}$ is said to be admissible if $\mu:[s,T]\to \pb(U)$,   $\nu:[s,T]\to \pb(U)$ are measurable such that
\begin{itemize}
  \item[$1^\circ$]  SDEs \eqref{l-1} and \eqref{l-2} admit a strong solution $(X_t,\La_t)$ with initial value $(X_s,\La_s)=(x,i)$.
  \item[$2^\circ$] $\mu_t$ and $\nu_t$ are adapted to the $\sigma$-fields $\mathscr{F}_t=\overline{\sigma\{(X_u,\La_u); s\leq u\leq t\}}$ for almost all $t\in [0,T]$. Here the over line in $\overline{\sigma\{(X_u,\La_u); s\leq u\leq t\}}$ means the completion of the $\sigma\{(X_u,\La_u); s\leq u\leq t\}$.
\end{itemize}
\end{mydefn}

Denote by $\Pi_{s,x,i}$ the collection of all admissible feedback controls with initial value $(X_s,\La_s)=(x,i)$ for $(s,x,i)\!\in \! [0,T)\times \R^d\times \S$. The class $\Pi_{s,x,i}$ contains many interesting controls, especially, it contains the path dependent feedback controls on the component $(X_t)$.  Due to Lemma \ref{app-2} in the Appendix, the condition $2^\circ$ in Definition \ref{def-1} that $\mu_t$ is adapted to $\overline{\sigma\{(X_u,\La_u); s\!\leq u\leq\! t\}}$ yields that there is a  measurable function $F_t$ such that $\mu_t=F_t(X_{[s,t]},\La_{[s,t]})$ almost surely. Therefore, condition $2^\circ$ in Definition \ref{def-1} ensures that the control policies $\mu_t$ and $\nu_t$ are indeed a kind of feedback control.
According to \cite[Theorem T46, p.68]{Mey}, for any measurable process adapted to the $\sigma$-fields $\F_t$, there exists a modification of this process progressively measurable with respect to the same family $\F_t$. Thus, it is enough to assume $\mu_t$ and $\nu_t$ to be adapted to the $\sigma$-fields generated by the process $(X_t,\La_t)$ in Definition \ref{def-1}.

\begin{myrem}\label{rem-1}
The request that condition $2^\circ$ of Definition \ref{def-1} holds for almost all $t$ has been used in Linquist \cite{Lin}. As a solution to SDEs \eqref{l-1}, \eqref{l-2} $($equivalently, \eqref{l-7} below$)$, the controlled system $(X_t,\La_t)$ remains the same when modifying $\mu_t$ and $\nu_t$ for $t\in [0,T]$ in a null set of the Lebesgure measure. So condition $2^\circ$ of Definition \ref{def-1} can be modified to require that $\mu_t$ and $\nu_t$ are adapted to $\F_t$ for every $t\in [0,T]$.
\end{myrem}

Haussmann and Suo \cite{HS95} assumed the existence of martingale solution of the corresponding stochastic dynamical system and proved the existence of optimal control which is not necessary a feedback control policy. In contrast to \cite{HS95}, some explicit conditions on the coefficients of the studied system \eqref{l-1} and \eqref{l-2} will be presented below to ensure the existence of strong solution of the studied system. By taking advantage of this property, we can show the existence of the optimal feedback controls.

Given two measurable functions $f:[0,T]\times \R^d\times \S\times \pb(U)\times \pb(U)\ra [0,\infty)$ and $g:\R^d\ra [0,\infty)$, the expected cost relative to the control $\alpha$ is defined by
\begin{equation}\label{l-4}
  J(s,x,i,\alpha)=\E\Big[\int_s^T f(t,X_t,\La_t,\mu_t,\nu_t)\d t+g(X_T)\Big].
\end{equation}
The corresponding value function is defined by
\begin{equation}\label{l-5}
  V(s,x,i)=\inf_{\alpha\in  \Pi_{s,x,i}} J(s,x,i,\alpha).
\end{equation}
An admissible control $\alpha^\ast\in   \Pi_{s,x,i}$ is called optimal, if it holds
\begin{equation}\label{l-6}
  V(s,x,i)=J(s,x,i,\alpha^\ast).
\end{equation}

The hypothesises on the coefficients of $(X_t,\La_t)$ are listed as follows in order to ensure the existence of strong solution $(X_t,\La_t)$ satisfying \eqref{l-1} and \eqref{l-2}. 
\begin{itemize}
  \item[(H1)] There exists a constant $C_1>0$ such that
  \begin{gather*}
  |b(x,i,\mu)-b(y,i,\nu)|^2+\|\sigma(x,i,\mu)-\sigma(y,i,\nu)\|^2\leq\! C_1\big(|x-y|^2+W_1(\mu,\nu)^2\big)
  \end{gather*} for $x,y\!\in\!\R^d, i\!\in\!\S,\mu,\,\nu\!\in \!\pb(U)$,
  where $|x|^2=\sum_{k=1}^d x_k^2$, $\|\sigma\|^2=\mathrm{tr}(\sigma\sigma')$, and $\sigma'$ denotes the transpose of the matrix $\sigma$.
  \item[(H2)] For every $x\in \R^d$, $\nu\in \pb(U)$, $(q_{ij}(x,\nu))$ is conservative, i.e. $q_i(x,\nu)=\sum_{j\neq i}q_{ij}(x,\nu)$ for every $i\in \S$. Moreover, $M:=\sup_{x\in\R^d,\nu\in\pb(U)}\max_{i\in\S}q_i(x,\nu)<\infty$.
  \item[(H3)] There exists a constant $C_2>0$ such that for every $i,\,j\in\S$, $x,\,y\in \R^d,\ \mu,\nu\in\pb(U)$,
  \begin{align*} |q_{ij}(x,\mu)-q_{ij}(y,\nu)|\leq C_2(|x-y|+W_1(\mu,\nu)).
      \end{align*}
  \item[(H4)] $U\subset \R^k$ is compact for some $k\in \N$.
\end{itemize}

Our first main result of this work is on the existence of the optimal feedback control.
\begin{mythm}\label{t1} Assume that (H1)-(H4) hold and $f:[0,T]\times \R^d\times \S\times \pb(U)\times \pb(U)\to \R$, and $g:\R^d\to \R$ are lower semicontinuous and bounded from below.
Then for every $(s,x,i)\in [0,T]\times\R^d\times \S$, there exists an optimal admissible control $\alpha^\ast\in  \Pi_{s,x,i}$ corresponding to the value function $V(s,x,i)$.
\end{mythm}

Note that the assumptions (H1)-(H3) ensure that SDEs \eqref{l-1} and \eqref{l-2} admit a unique strong solution given $\mu_t\equiv \mu$ and $\nu_t\equiv \nu$ in $\pb(U)$, which are also used to show the tightness of the distributions of $(X_t^{(n)})_{n\geq 1}$ in the proof of Theorem \ref{t1}. The Lipschitz conditions can be replaced by some non-Lipschitz conditions to ensure the existence of strong solutions for such kind of system.  See, for instance, \cite{Sh15c} for the existence of strong solutions of state-dependent regime-switching processes from the viewpoint of SDEs, and \cite{SWY} for the existence of strong solution of stochastic functional differential equations under non-Lipschitz conditions.

\subsection{Proof of Theorem \ref{t1}}

Before proving Theorem \ref{t1}, we make some necessary preparations.
Let
$\pb(U)$ be endowed with the $L^1$-Wasserstein distance.
$\mathcal{C}([0,T];\R^d)$ is endowed with the uniform topology, and $ \mathcal D([0,T];\pb(U))$, $\mathcal D([0,T];\S)$ are endowed  with pseudopath topology which makes $ \mathcal D([0,T];\pb(U))$ and $\mathcal{D}([0,T];\S)$ to be Polish spaces (See, for instance, \cite[Theorem 5.6, p.121]{EK}).
Let
\begin{equation}\label{sp-u}
\mathscr{U}=\{\mu:[0,T]\to \pb(U) \ \text{is Borel measurable}\}.
\end{equation}
We view $\mathscr{U}$ as a subspace of $\pb([0,T]\times U)$ through the map
 \[(\mu_t)_{t\in[0,T]}\mapsto \bar \mu,\]
where $\bar \mu$ is defined as follows: for $A\in \mathscr{B}([0,T])$, $B\in \mathscr{B}(U)$, define
\[\bar\mu(A\times B)=\frac 1T\int_A\mu_t(B)\d t.\]
Endow $\mathscr{U}$ with the weak topology, which is equivalent to the topology induced from the $L^1$-Wasserstein distance in $\pb([0,T]\times U)$ defined as
\[W_1(\bar \mu,\bar \nu)=\inf_{\Gamma\in \C(\bar\mu,\bar \nu)}\int_{([0,T]\!\times\!U)^2}\!\big(|s-t|+|x-y|\big)\,\Gamma((\d s,\d x),(\d t,\d y)).\]
due to the boundedness of $[0,T]\times U$. Moreover, since $[0,T]\times U$ is compact,  the space $\mathscr{U}$, as a closed set of the compact space $\pb([0,T]\times U)$, is also compact  (cf. \cite{AGS, Vi}).

Let
\[\mathcal Y= \mathcal{C}([0,T];\R^d)\times \mathcal D([0,T];\S)\times \mathscr{U}\times \mathscr{U},\]
and $\wt{\mathcal Y}$ the Borel $\sigma$-field, $\wt {\mathcal{Y}}_t$ the $\sigma$-fields up to time $t$.
Then, as a product space endowed with the product topology, $\mathcal Y$ is a Polish space.

In the argument of Theorem \ref{t1}, we shall consider the tightness of the distributions of admissible controls by transforming them into the canonical space $\mathcal Y$ via  a measurable map $\Psi$. For an admissible control $\alpha=(\mu_\cdot,\nu_\cdot)$ in $\Pi_{s,x,i}$, $\Psi_\alpha:\Omega\ra \mathcal Y$ is defined by
\[\Psi_\alpha(\omega)=(X_t(\omega),\La_t(\omega),\mu_t(\omega),\nu_t(\omega))_{t\in [0,T]}.\]
Here,  $X_r(\omega):=x,\ \La_r(\omega):=i,\mu_r(\omega):=\mu_s$, and $\nu_r(\omega):=\nu_s$ for $ r\in  [0,s]$.  Let $R=\p\circ \Psi_\alpha^{-1}$ be the corresponding probability measure on $\mathcal Y$ associated with the control $\alpha=(\mu_\cdot,\nu_\cdot)$.
Let
\begin{equation}\label{RR0}
\mathcal{R}_{s,x,i}=\big\{R=\p\circ\Psi_\alpha^{-1}; \alpha\in \Pi_{s,x,i}\big\}.
\end{equation}



As a preparation, we introduce Skorokhod's representation of $(\La_t)$ in terms of the Poisson random measure as in \cite[Chapter II-2.1]{Sk89} or \cite{YZ}. For each $x\in \R^n$ and $\nu\in \pb(U)$, we construct a family of intervals $\{\Gamma_{ij}(x,\nu); \ i,j\in \S\}$ on the half line in the following manner:
\begin{align*}
  \Gamma_{12}(x,\nu)&=[0,q_{12}(x,\nu))\\
  \Gamma_{13}(x,\nu)&=[q_{12}(x,\nu),q_{12}(x,\nu)+q_{13}(x,\nu))\\
  \ldots&\ldots\ldots\\
  \Gamma_{1N}(x,\nu)&=\big[\sum_{j=1}^{N-1}q_{1j}(x,\nu), q_{1}(x,\nu)\big)\\
  \Gamma_{21}(x,\nu)&=[q_1(x,\nu),q_1(x,\nu)+q_{21}(x,\nu))\\
  \Gamma_{23}(x,\nu)&=[q_1(x,\nu)+q_{21}(x,\nu),q_1(x,\nu)+q_{21}(x,\nu)+q_{23}(x,\nu))
\end{align*}
and so on. Therefore, we obtain a sequence of consecutive, left-closed, right-open intervals $\Gamma_{ij}(x,\nu)$, each having length $q_{ij}(x,\nu)$. For convenience of notation, we set $\Gamma_{ii}(x,\nu)=\emptyset$ and $\Gamma_{ij}(x,\nu)=\emptyset$ if $q_{ij}(x,\nu)=0$.
Define a function $\vartheta:\R^n\times \S\times \pb(U)\times \R\ra \R$ by
\[\vartheta(x,i,\nu,z)=\sum_{l\in \S} (l-i)\mathbf 1_{\Gamma_{il}(x,\nu)}(z).\]
Then the process $(\La_t)$ can be expressed by the following SDE
\begin{equation}\label{l-7}
\d \La_t=\int_{[0,H]} \vartheta(X_t,\La_{t-},\nu_{t-},z)N_1(\d t,\d z),
\end{equation}
where $H=N(N-1) M$, $N_1(\d t,\d z)$ is a Poisson random measure with intensity $\d t\times \mathbf{m}(\d z)$, and $\mathbf{m}(\d z)$ is the Lebesgue measure on $[0,H]$. Here we also assume that the Poisson random measure $N_1$ and the Brownian motion $(B_t)$ are mutually independent.
Let $p_1(t)$ be the stationary point process corresponding to the Poisson random measure $N_1(\d t,\d z)$. Due to the finiteness of $\mathbf{m}(\d z)$ on $[0, H]$, there is only finite number of jumps of the process $p_1(t)$ in each finite time interval. Let $ 0=\varsigma_0<\vsig_1<\ldots<\vsig_n<\ldots$ be the enumeration of all jumps of $p_1(t)$. It holds that $\lim_{n\ra \infty}\vsig_n=+\infty$ almost surely. Due to \eqref{l-7}, it follows that, if $\La_0=i$,
\begin{equation} \label{l-8}
 \La_{\vsig_1}=i+\sum_{l\in\S}(l-i) \mathbf{1}_{\Gamma_{il}(X_{\vsig_1},\nu_{\vsig_1})}(p_1(\vsig_1)).
\end{equation}
This yields that $(\La_t)$ has a jump at $\vsig_1$ (i.e. $\La_{\vsig_1}\neq \La_{\vsig_1-}$) if $p_1(\vsig_1)$ belongs to the interval $\Gamma_{il}(X_{\vsig_1},\nu_{\vsig_1})$ for some $l\neq i$. At any other cases, $(\La_t)$ admits no jump at $\vsig_1$. So the set of jumping times of $(\La_t)$ is a subset of $\{\vsig_1,\vsig_2,\ldots\}$. This fact will be used below without mentioning it again.

\noindent \textbf{Proof of Theorem \ref{t1}}

  If $V(s,x,i)=\infty$, then according to the definition of $V$, any admissible control $\alpha$ will be optimal. Hence, we only need to consider the case $V(s,x,i)<\infty$. To simplify the notation, we consider only $s=0$, and more general cases for $s\in (0,T]$ can be proved in the same way with suitable modification. The proof is separated into three steps.

  \textbf{Step 1}. In this step we show the tightness of a minimizing sequence.

There exists a sequence of admissible controls $\alpha_n=(\mu_\cdot^{(n)},\nu_\cdot^{(n)})$ in $\Pi_{0,x,i}$ such that
  \begin{equation}\label{m-1}
    \lim_{n\ra \infty} J(0,x,i,\alpha_n)=V(0,x,i)<\infty.
  \end{equation}
Denote by $(X_t^{(n)},\La_t^{(n)})$  the controlled system associated with $\alpha_n$.
Let $R_n$, $n\geq 1$, be the joint distribution of $(X_t^{(n)},\La_t^{(n)},\mu_t^{(n)}$, $\nu_t^{(n)})_{t\in[0,T]}$, which is a sequence of probability measures in the canonical space $\mathcal Y$.  In this step we aim to prove the tightness of $(R_n)_{n\geq 1}$. Denote respectively by $\ll_X^n$, $\ll_\La^n$,  $\ll_\mu^n$, and $\ll_\nu^n$ the marginal distribution of $R_n$ for $n\geq 1$.   Since $\ll_\mu^n$ and $\ll_\nu^n$ are located in the compact set $\mathscr{U}$, we do not need to consider the marginal distributions $\ll_\mu^n$ and $\ll_\nu^n$ for the tightness of $(R_n)_{n\geq 1}$.

We first prove that $(\ll_\La^n)_{n\geq 1}$ is tight by using Kurtz's tightness criterion
(cf. \cite[Theorem 8.6, p.137]{EK}). As $\S$ is a finite set, we only need to show there exists a sequence of nonnegative random variable $\gamma_n(\delta)$ such that
  \begin{equation}\label{m-5}
    \E \big[\mathbf{1}_{\La_{t+u}^{(n)}\neq \La_t^{(n)}}\big|\F_t\big]\leq \E \big[\gamma_n(\delta)\big|\F_t\big],\quad 0\leq t\leq T,\ 0\leq u\leq \delta,
  \end{equation}  and $\lim_{\delta\downarrow 0}\sup_n\E[\gamma_n(\delta)]=0$. Due to (H2), the boundedness of $(q_{ij}(x,\nu))$ implies
  \begin{equation*}
    \begin{split}
       \p(\La_r^{(n)}= \La_t^{(n)},\ \forall \,r\in [t,t+u]) & \geq \E \big[\exp\big(-\sup_{x\in\R^d, \nu\in \pb(U)}\max_{j\in \S} q_j(x,\nu) u\big)\big] \\
         & \geq \exp(-M u).
    \end{split}
  \end{equation*}
  Then, for every $0\leq u\leq \delta$,
  \begin{equation}\label{m-6}
    \begin{split}
       \E \big[\mathbf{1}_{\La_{t+u}^{(n)}\neq \La_t^{(n)}}\big|\F_t\big] & \leq 1-\p\big(\La_r^{(n)}=\La_t^{(n)},\ \forall\,r\in [t,t+u]\big) \\
         & \leq 1-\e^{-M \delta}=:\gamma_n(\delta).
    \end{split}
  \end{equation}
  It is clear that $\lim_{\delta\downarrow 0} \sup_n\E [\gamma_n(\delta)]=0$ and \eqref{m-5} is verified. We conclude that $(\ll_{\La}^n)_{n\geq 1}$ is tight.


Since $U$ is compact,  $(\pb(U),W_1)$ is a compact Polish space (cf. e.g.  \cite{AGS}). This implies that the diameter of $\pb(U)$ is finite. Namely, there exists a constant $K>0$ such that \[\mathrm{Diam}(\pb(U)):=\sup_{\mu,\,\nu\in \pb(U)}W_1(\mu,\nu)\leq K.\]
Hence, the global Lipschitz condition (H1) implies the linear growth condition, i.e. there exists a $C>0$ such that $|b(x,i,\mu)|+\|\sigma(x,i,\mu)\|\leq C(1+|x|)$ for every $i\in\S$, $\mu\in \pb(U)$, which leads to
\begin{equation}\label{m-3}
  \E\big[\sup_{0\leq t\leq T}|X_t^{(n)}|^p\big]\leq C(T,x,p), \quad n\geq 1,\ p\geq1,
\end{equation} where $C(T,x,p)$ is a constant depending on $T,\, x,\,p$ (cf.   \cite[Lemma 3.1, p.28]{MY}).

By It\^o's formula, for $0\leq t_1<t_2\leq T$,
\begin{align*}
    &\E|X_{t_2}^{(n)}-X_{t_1}^{(n)}|^4\\
    &\leq 8\E \Big|\int_{t_1}^{t_2}b(X_r^{(n)},\La_r^{(n)},\mu_r^{(n)})\d r\Big|^4+8\E \Big|\int_{t_1}^{t_2}\sigma(X_r^{(n)},\La_r^{(n)},\mu_r^{(n)})\d B_r\Big|^4\\
    &\leq 8(t_2-t_1)^3\E \int_{t_1}^{t_2}|b(X_r^{(n)},\La_r^{(n)},\mu_r^{(n)})|^4\d r+288(t_2-t_1)\E \int_{t_1}^{t_2}|\sigma(X_r^{(n)},\La_r^{(n)},\mu_r^{(n)})|^4\d r\\
    &\leq C(t_2-t_1)\int_{t_1}^{t_2}\big(1+\E |X_r^{(n)}|^4\big)\d r.
\end{align*}
Applying condition (H1) again, we have
$\dis \int_0^T\!\!\!\E |X_r^{(n)}|^4\d r\leq C$  for some constant $C$, independent of $n$, (cf. \cite[Theorem 3.20]{MY}). Furthermore, invoking the fact $X^{(n)}_0=x$, we conclude that $(\ll_X^n)_{n\geq 1}$ is tight by virtue of
\cite[Theorem 12.3]{Bill}.

\textbf{Step 2}. This step is to show that the limit of $(R_n)$ is also an admissible feedback control, which will be showed to be the desired optimal feedback control in  step 3.

Because all the marginal distributions of $R_n$, $n\geq 1$ are tight, we get $R_n$, $n\geq 1$ is tight as well. Indeed, for any $\veps>0$, there exist compact subsets $K_1\subset \mathcal{C}([0,T];\R^n)$, $K_2\subset \mathcal D([0,T];\S)$, and $K_3, K_4=\mathscr{U}$ such that for every $n\geq 1$,
\[\min\{\ll_B^n(K_1), \ll_X^n(K_2), \ll_\La^n(K_3),\ll_\mu^n(K_4)\}\geq 1-\veps.\]
This yields that
\[R_n(K_1\times K_2\times K_3\times K_4)\geq 1-\ll_X^n(K_1^c)-\ll_\La^n(K_2^c)-\ll_\mu^n(K_3^c)-\ll_{\nu}^n(K^c_4)\geq 1-4\veps.\]
So $(R_n)_{n\geq 1} $ is tight.

As a consequence of the tightness of $(R_n)_{n\geq 1}$, up to extracting a subsequence, we have $R_n$ converges weakly to some probability measure $R_0$ on $\mathcal Y$. Since $\mathcal Y$ is a Polish space, according to Skorokhod's representation theorem (cf. \cite{EK}, Theorem 1.8, p.102), there exists a probability space $(\Omega',\F',\p')$ on which defined a sequence of random variables $Y_n=(X_t^{(n)},\La_t^{(n)},\mu_t^{(n)},\nu_t^{(n)})_{t\in [0,T]}\in \mathcal Y$, $n\geq 0$, with the distribution $R_n$, $n\geq 0$, respectively such that
\begin{equation}\label{m-7}
\lim_{n\ra \infty} Y_n=Y_0,\quad \p'\text{-a.s.}.
\end{equation}
In this step we want to show that $Y_0=(X_t^{(0)}, \La_t^{(0)}, \mu_t^{(0)}, \nu_t^{(0)})$ is also associated with an admissible control.

For $0\leq t_1<t_2<\ldots<t_k\leq T$, define the projection map $\pi_{t_1\ldots t_k}:\mathcal D([0,T];\S)\ra \S^k$ by
\[\pi_{t_1\ldots t_k}(\La_\cdot)=(\La_{t_1},\ldots,\La_{t_k}).\]
Let $\mathcal T_0$ consist of those $t\in [0,T]$ for which the projection $\pi_t:\mathcal D([0,T];\S)\ra \S$ is continuous except at points from a set of $R_0$-measure 0. For $t\in [0,T]$, $t\in \mathcal T_0$ if and only if $R_0(J_t)=0$, where
\[J_t=\{\La_\cdot\in \mathcal D([0,T];\S); \La_t\neq \La_{t-}\}.\]
Also, $0,\,T\in \mathcal{T}_0$ by convention.
It is known that the complement of $\mathcal T_0$ in $[0,T]$  is at most countable (cf. \cite[p. 124]{Bill}).
So, for every bounded continuous function $h$ on $\S$,
\begin{equation}\label{con-1}
\lim_{n\ra \infty}\int_s^t h(\La_r^{(n)})\d r =\int_s^t h(\La_r^{(0)})\d r,\quad 0\leq s<t\leq T,\ \p'\text{-a.s.}.
\end{equation}

Now consider the following convergence:  for any bounded continuous function $\Phi$ on $\mathscr{P}(U)$, it holds
\begin{equation}\label{con-2}
\lim_{n\to \infty} \int_s^t \Phi(\mu_r^{(n)})\d r=\int_s^t \Phi(\mu_r^{(0)})\d r,\ 0\leq s<t\leq T,\ \p'\text{-a.s.}.
\end{equation}
To this aim, for any $m\in \N$, take a continuous function $\beta_m:[0,T]\to [0,1]$ such that $\mathrm{Leb}\big\{r\in [0,T];|\beta_m(r)-\mathbf{1}_{[s,t]}(r)|>0\big\}\leq 1/(mK)$, where $K=\sup_{\mu\in \pb(U)}|\Phi(z)|<\infty$ and $\mathrm{Leb}$ denotes the Lebesgure measure over $\R$.
Then
\begin{equation}\label{cov-1}
\begin{split}
  &\Big|\int_{s}^t\Phi(\mu_r^{(0)})\d r-\int_0^T\beta_m(r)\Phi(\mu_r^{(0)})\d r\Big|\leq \frac 1m,\\
  &\Big|\int_{s}^t\Phi(\mu_r^{(n)})\d r-\int_0^T\beta_m(r)\Phi(\mu_r^{(n)})\d r\Big|\leq  \frac 1m,\quad \forall\, n\geq 1.
\end{split}
\end{equation}
Since $(r,u)\mapsto \beta_m(r)\Phi(u)$ is bounded and continuous, it follows from the weak convergence of $\mu_r^{(n)}(\d u)\d r$ to $\mu_r^{(0)}(\d u)\d r$ in $\mathscr{U}$ that there exists $M_1\in\N$ such that for any $n\geq M_1$
\begin{equation}\label{cov-2}
\Big|\int_0^T\beta_m(r)\Phi(\mu_r^{(n)})\d r-\int_0^T\beta_m(r)\Phi(\mu_r^{(0)})\d r\Big|\leq \frac 1m,\quad \p'\text{-a.s.}.
\end{equation}
Combining \eqref{cov-1} and \eqref{cov-2} together, we obtain that there exists $M_1\in \N$ such that for any $n\geq M_1$,
\[\Big|\int_s^t \Phi(\mu_r^{(n)})\d r-\int_s^t\Phi(\mu_r^{(0)})\d r\Big|\leq \frac{3}m,\quad \p'\text{-a.s.}\]
which yields the convergence \eqref{con-2}  by passing $m$ to $\infty$.

According to the definition of stochastic integral with respect to the Brownian motion, it is easy to see that \eqref{con-1} and \eqref{con-2} still hold by replacing $\d r$ with $\d B_r$ up to taking some subsequence if necessary.

Combining \eqref{con-1}, \eqref{con-2} with the almost sure convergence of
$Y_n$ to $Y_0$, 
by passing $n$ to $\infty$ in the equation
\begin{equation*}
  X_t^{(n)}=x+\int_0^tb(X_r^{(n)},\La_r^{(n)},\mu_r^{(n)})\d r+\int_0^t\sigma(X_r^{(n)},\La_r^{(n)},\mu_r^{(n)})\d B_r,
\end{equation*}
we obtain that
\begin{equation}\label{m-8}
  X_t^{(0)}=x+\int_0^tb(X_r^{(0)},\La_r^{(0)},\mu_r^{(0)})\d r+\int_0^t\sigma(X_r^{(0)},\La_r^{(0)},\mu_r^{(0)})\d B_r.
\end{equation}

In terms of Skorokhod's representation \eqref{l-7} for jumping process $(\La_t^{(n)})$, we have
\begin{equation}\label{m-9}
  \La_t^{(n)}=i+\int_0^t\int_{[0,H]}\!\vartheta(X_r^{(n)},\La_{r-}^{(n)},\nu_{r-}^{(n)},z)N_1(\d r,\d z).
\end{equation}
Since $(x,\nu)\mapsto q_{ij}(x,\nu)$ is continuous for every $i,j\in \S$, one gets that $\mathbf{1}_{\Gamma_{ij}(y,\nu')}(z)$ tends to $\mathbf{1}_{\Gamma_{ij}(x,\nu)}(z)$ as $|y-x|\ra 0$ and $W_1(\nu',\nu)\ra 0$. Similar to \eqref{con-2}, using continuous functions to approximate the indicator function $\mathbf{1}_{[0,t]}$ uniformly w.r.t. $n$, we obtain from \eqref{m-9} by passing $n\ra \infty$ that
\begin{equation}\label{m-10}
  \La_t^{(0)}=i+\int_0^t\int_{[0,H]}\!\vartheta(X_r^{(0)},\La_{r-}^{(0)},\nu_{r-}^{(0)},z)N_1(\d r,\d z).
\end{equation}
By Skorokhod's representation \eqref{l-7}, this yields that
\[
\p\big(\La_{t+\delta}^{(0)}=j|\La_t^{(0)}=i,X_t^{(0)}=x,\nu_t^{(0)}=\nu\big)=\begin{cases}
  q_{ij}(x,\nu)\delta+o(\delta), &\mbox{if $i\neq j$},\\
  1+q_{ii}(x,\nu)\delta+o(\delta),&\mbox{otherwise},
\end{cases}
\] provided $\delta>0$. Moreover, there is no $t_0\in [0,T]$ such that $\p'(\La_{t_0}^{(0)}\neq \La_{t_0-}^{(0)})>0$, which means that $\mathcal T_0=[0,T]$. Hence, $\lim_{n\ra \infty} \La_t^{(n)}=\La_t^{(0)}$ $\p'$-a.s. for every $t\in [0,T]$.

Till now what is left is to show $\mu_t^{(0)}$ and $\nu_t^{(0)}$ are adapted to the $\sigma$-fields generated by $(X_r^{(0)},\La_r^{(0)})$ up to time $t$.
To this aim, we adopt the notation in the study of backward martingale to define
  $$\F_{-n, t}^{X,\La}=\overline{\sigma\{(X_r^{(m)},\La_r^{(m)});\ m\geq n, r\in [0,t]\}}.$$  Then
\[\F_{-1,t}^{X,\La}\supset\F_{-2,t}^{X,\La}\supset\cdots\supset \F_{-n,t}^{X,\La}\supset\F_{-n-1,t}^{X,\La}\supset\cdots.\] Put $\F_{-\infty,t}^{X,\La}=\bigcap_{n\geq 1}\F_{-n,t}^{X,\La}$.  $\F_{-\infty,t}^{X,\La}$ is easily checked to be a $\sigma$-field which concerns only the limit behavior of the sequence $(X_r^{(n)},\La_r^{(n)})_{r\in [0,t]}$ as $n$ tends to $\infty$. Moreover, since $\lim_{n\ra \infty}\La^{(n)}_t=\La_t^{(0)}$ and $\lim_{n\ra \infty}X_t^{(n)}=X_t^{(0)}$ a.s. for every $t\in [0,T]$,  it holds \[\F_{-\infty}^{X,\La}=\overline{\sigma\{(X_r^{(0)},\La_r^{(0)}); \ r\in [0,t] \}}.\]

Define $\F_{-n,t}^\mu=\overline{\sigma\{\mu_t^{(m)}; m\geq n\}}$. Due to  Definition \ref{def-1}$(2^\circ)$, $\mu_t^{(n)}$ is in $\F_{-n}^{X,\La}$ for each $n\geq 1$, and hence $\F_{-n,t}^\mu\subset \F_{-n,t}^{X,\La}$. Therefore, according to Lemma \ref{lem-lusin} below, the fact $\lim\limits_{n\ra \infty}W_1(\mu_t^{(n)},\mu_t^{(0)})=0$ a.s. yields that
\[\overline{\sigma\{\mu_t^{(0)}\}}\subset \bigcap_{n\geq 1}\F_{-n,t}^\mu\subset \bigcap_{n\geq 1}\F_{-n,t}^{X,\La}= \F_{-\infty,t}^{X,\La}=\overline{\sigma\{(X_r^{(0)},\La_r^{(0)});\ r\in [0,t]\}}.\]
This means that $\mu_t^{(0)}$ is adapted to $\overline{\sigma\{(X_r^{(0)},\La_r^{(0)}); r\in [0,t]\}}$ for almost all $t\in [0,T]$. Similarly, we can show that $\nu_t^{(0)}$ is also adapted to $\F_t^{(0)}$ for almost all $t\in [0,T]$.
Joining this with \eqref{m-8}, \eqref{m-10}, we finally show that
that $\alpha_0:=(\mu_t^{(0)},\nu_t^{(0)})$ associated with $Y_0=(X_t^{(0)}, \La_t^{(0)}, \mu_t^{(0)}, \nu_t^{(0)})$ is  an admissible feedback control in $\Pi_{0,x,i}$.

\textbf{Step 3}. By \eqref{m-1} and the lower semicontinuity of $f$ and $g$, we have
\begin{align*}
  V(0,x,i)&=\lim_{n\ra \infty} J(0,x,i,\alpha_n)\\
  & =\lim_{n\ra \infty}\E_{\p'}\Big[\int_0^T\!f(t,X_t^{(n)},\La_t^{(n)},\mu_t^{(n)},\nu_t^{(n)})\d t+g(X_T^{(n)})\Big]\\
  &\geq \E_{\p'}\Big[\int_0^T f(t, X_t^{(0)},\La_t^{(0)},\mu_t^{(0)},\nu_t^{(0)})\d t+g(X_T^{(0)})\Big]\\
  &=J(0,x,i,\alpha_0)\\
  &\geq V(0,x,i).
\end{align*}
Therefore, $\alpha_0$ is an optimal admissible feedback control. The proof of this theorem is complete.
\fin

\begin{mylem}\label{lem-lusin}
Let $\mu_t^{(n)},\,\nu_t^{(n)}$, $\mu_t^{(0)},\,\nu_t^{(0)}$ be given in the argument of Theorem \ref{t1}. Then for almost all $t\in[0,T]$,
\[\lim_{n\to \infty} W_1(\mu_t^{(n)},\mu_t^{(0)})=0,\quad \lim_{n\to \infty} W_1(\nu_t^{(n)},\nu_t^{(0)})=0,\quad \p'\text{-a.s.}.\]
\end{mylem}

\begin{proof}
  According to the $\p'$-a.s. convergence of $(\mu_t^{(n)})$ to $(\mu_t^{(0)})$ in $\mathscr{U}$, we obtain that for any bounded continuous functions $\beta:[0,T]\to \R$, $\phi:U\to \R$,
  \begin{equation}\label{el-1}\lim_{n\to \infty}\int_0^T\int_{U}\beta(r)\phi(u)\mu_r^{(n)}(\d u)\d r=\int_0^T\int_U\beta(r)\phi(u)\mu_r^{(0)}(\d u)\d r,\ \ \p'\text{-a.s.}.
  \end{equation}
  For any $A\in \mathscr{B}([0,T])$, denote by $\mathbf{1}_A$ its indicator function. According to Lusin's theorem, for each $m\in \N$ there exists a continuous function $\beta_m$ on $[0,T]$ such that
  \[\mathbf{m}(\{\mathbf1_A\neq \beta_m\})\leq \frac1{m}\] and $\sup_{t\in [0,T]} |\beta_m(t)|\leq \sup_{t\in[0,T]}|\mathbf1_{A}(t)|=1$, where $\mathbf{m}(\,\cdot\,)$ denotes the Lebesgue measure on $[0,\infty)$. For any bounded continuous function $\phi:U\to \R$ with $K:=\sup_{u\in U}|\phi(u)|<\infty$,
  \begin{align*}
    &\Big|\int_0^T\!\!\int_U\mathbf{1}_A(r)\phi(u)\mu_r^{(n)}(\d u)\d r-\int_0^T\!\!\int_U\mathbf 1_A(r)\phi(u)\mu_r^{(0)}(\d u)\d r\Big|\\
    &\leq \Big|\int_0^T\!\!\int_U\big(\mathbf1_{A}(r)\!-\!\beta_m(r)\big)\phi(u)\mu_r^{(n)}(\d u)\d r\Big|\!+\!\Big|\int_0^T\!\!\int_U\!\big(\beta_m(r)\!-\!\mathbf1_A(r)\big)
    \phi(u)\mu_r^{(0)}(\d u)\d r\Big|\\
    &\quad+\Big|\int_0^T\!\!\int_U\beta_m(r)\phi(u)\mu_r^{(n)}(\d u)\d r\!-\!\int_0^T\!\!\int_U\!\beta_m(r)\phi(u)\mu_r^{(0)}(\d u)\d r\Big|\\
    &\leq \frac{4K}{m}\!+\!\Big|\int_0^T\!\!\int_U\beta_m(r)\phi(u)\mu_r^{(n)}(\d u)\d r\!-\!\int_0^T\!\!\int_U\!\beta_m(r)\phi(u)\mu_r^{(0)}(\d u)\d r\Big|
  \end{align*}
  Invoking \eqref{el-1}, for every $\veps>0$, taking first $m\in \N$ large enough then $n\in \N$ large enough, we can get
  \[\Big|\int_0^T\!\!\int_U\mathbf{1}_A(r)\phi(u)\mu_r^{(n)}(\d u)\d r-\int_0^T\!\!\int_U\mathbf 1_A(r)\phi(u)\mu_r^{(0)}(\d u)\d r\Big|\leq \veps,\quad \p'\text{-a.s.}.\]
  Namely,
  \begin{equation}\label{el-2}
  \lim_{n\to \infty} \int_A\int_U\phi(u)\mu_r^{(n)}(\d u)\d r=\int_A\int_U\phi(u)\mu_r^{(0)}(\d u)\d r \quad \p'\text{-a.s.}.
  \end{equation}
  Due to the arbitrariness of $A$ and $\phi$, we obtain that for $\mathbf{m}$-a.e. $t$, $\mu_t^{(n)}$ converges weakly to $\mu_t^{(0)}$ $\p'$-a.s..  Since $U$ is a compact set, the topology induced by the Wasserstein distance $W_1$ is equivalent to the weak topology on $\pb(U)$, and hence the desired conclusion follows immediately. The corresponding result for $\nu_t^{(n)}$ converging to $\nu_t^{(0)}$ can be proved in the same way.
\end{proof}

\section{Dynamic programming principle}

In this section, we go to establish the dynamic programming principle for the optimal control problem associated with the value function $V(s,x,i)$.

Let us begin with the discussion on the continuity of the value function after introducing some necessary notations. For $(s,x,i)\in [0,T]\!\times \!\R^d\!\times\!\S$, define the set of optimal controls by
\begin{equation}\label{k-1}
\Pi_{s,x,i}^0=\big\{\alpha\in  \Pi_{s,x,i};\, J(s,x,i,\alpha)=  V(s,x,i)\big\}.
\end{equation}
Similarly, define
\begin{equation}\label{RR1}
\mathcal{R}_{s,x,i}^0=\big\{R=\p\circ\Psi_\alpha^{-1}; \alpha\in \Pi_{s,x,i}\ \text{is optimal}\big\}.
\end{equation}

As a product space, $\mathcal{Y}$ is a Polish space. Then $\pb(\mathcal Y)$ is also a Polish space endowed with $L^1$-Wasserstein distance $W_1$, which is defined as follows: for any $\tilde R_1$ and $\tilde R_2$ in $\pb(\mathcal Y)$, define
\[W_{1,\mathcal Y}(\tilde R_1,\tilde R_2)=\inf_{\Gamma\in \C(\tilde R_1,\tilde R_2)} \Big\{ \int_{\mathcal Y\times \mathcal Y}\rho\big((x_\cdot,\La_\cdot,\mu_\cdot,\nu_\cdot),
(x'_\cdot,\La'_\cdot,\mu'_\cdot,\nu'_\cdot)\big)\d \Gamma\Big\},\]
where
\[\rho\big((x_\cdot,\La_\cdot,\mu_\cdot,\nu_\cdot),
(x'_\cdot,\La'_\cdot,\mu'_\cdot,\nu'_\cdot)\big)=\|x_\cdot-x'_\cdot\|_{\infty}\!+\!  d(\La_\cdot,\La'_\cdot)\!+\!W_1(\bar\mu,\bar\mu')+W_1(\bar\nu,\bar\nu'),\]
where
$\|\cdot\|_\infty$ is the uniform norm on $C([0,T];\R^d)$, $d(\La_\cdot,\La'_\cdot)$ is the metric on $\D([0,T];\S)$ which makes it to be a Polish space (see \cite{Bill} for concrete construction), $\bar\mu$ is in $\pb([0,T]\times U)$ corresponding to $\mu_\cdot\in \mathscr U$ and $W_1$ is the associated $L^1$-Wasserstein distance on it.
As a subset of $\pb(\mathcal Y)$, $\mathcal R^0_{s,x,i}$ is closed under the metric $W_{1,\mathcal Y}$ when the cost functions $f$ and $g$ in \eqref{l-4} are lower semicontinuous. Analogous to the argument of Theorem \ref{t1}, we can show that $\mathcal R^0_{s,x,i}$ is tight. By Prohorov's theorem, $\mathcal R^0_{s,x,i}$ is a compact set in $\pb(\mathcal Y)$.
Moreover, we can rewrite the value function in the form
\begin{align*}
V(s,x,i)&=\inf_{R\in \mathcal{R}^0_{s,x,i}} J(s,x,i,R)\\
  &=\inf_{R\in \mathcal{R}^0_{s,x,i}}\E_R \Big[\int_0^T f(t,X_t,\La_t,\mu_t,\nu_t)\d t+g(X_T)\Big].
\end{align*}

We shall use the idea of Bertsekas and Shreve \cite[Proposition 7.32]{BS} to investigate the continuous property of the value function.

\begin{mythm}\label{t3}
$\mathrm{(a)}$\ Assume that $f:[0,T]\times \R^d\times \S\times \pb(U)\times \pb(U)\to \R$, and $g:\R^d\to \R$ are lower semicontinuous, then the value function $V$ is lower semicontinuous.

$\mathrm{(b)}$\ Assume that $f:[0,T]\times \R^d\times \S\times \pb(U)\times \pb(U)\to \R$, and $g:\R^d\to \R$ are continuous, then the value function $V$ is continuous.
\end{mythm}

\begin{proof}
  $\mathrm{(a)}$\ To show $V$ is lower semicontinuous, let $(s_n,x_n,i)\in [0,T]\times \R^d\times \S$ be such that
  $(s_n,x_n,i)\to (s_0,x_0,i)$. According to Theorem \ref{t1}, there exists a sequence of probability measures $R_n\in \mathcal{R}_{s_n,x_n,i}^0$ such that
  \begin{align*}
    V(s_n,x_n,i)=J(s_n,x_n,i,R_n)=\E_{R_n}\Big[\int_{s_n}^T f(t,X_t,\La_t,\mu_t,\nu_t)\d t+g(X_T)\Big],
  \end{align*} for $n\geq 1$.
  There exists a subsequence of $R_n$, call it $R_{n_k}\in \mathcal{R}_{s_{n_k},x_{n_k},i}$, such that
  \begin{align*}
    \liminf_{n\to \infty} V(s_n,x_n,i)=\lim_{k\to \infty} V(s_{n_k},x_{n_k},i)=\lim_{k\to \infty} J(s_{n_k},x_{n_k},i, R_{n_k}).
  \end{align*}
  Following the same procedure as in the argument of Theorem \ref{t1}, we can show that $(R_{n_k})_{k\geq 1}$ is tight. Hence there exists some $R_0\in \mathcal{R}_{s_0,x_0,i}$, at which the sequence $(R_{n_k})_{k\geq 1}$ accumulates.
  By the lower semicontinuity of $f$ and $g$, we have
  \begin{align*}
    \liminf_{n\to \infty} V(s_n,x_n,i)&=\lim_{k\to \infty} J(s_{n_k},x_{n_k},i, R_{n_k})\\
    &=\lim_{k\to \infty} \E_{R_{n_k}}\Big[\int_{s_{n_k}}^T f(t,X_t,\La_t,\mu_t,\nu_t)\d t+g(X_T)\Big]\\
    &\geq \E_{R_0}\Big[\int_{s_0}^T f(t,X_t,\La_t,\mu_t,\nu_t)\d t+g(X_T)\Big]\\
    &\geq V(s_0,x_0,i).
  \end{align*}
  So the value function is lower semicontinuous.

$\mathrm{(b)}$\ We only need to show $V$ is upper semicontinuous in this situation due to (a). Denote by $Z=[0,T]\times \R^d\times \S$ to
  simplify the notation. Let
  \[\mathcal{R}^0=\bigcup_{(s,x,i)\in Z} \mathcal{R}_{s,x,i}^0,\]
  and endow it with the metric $W_{1,\mathcal{Y}}$. The function $J$ can be extended naturally to $\mathcal{R}^0$ by putting
  \[J(s,x,i,R)=+\infty,\quad \text{if $R\not\in \mathcal{R}^0_{s,x,i}$}.\]
  Let $\mathrm{proj}_Z(\,\cdot\,)$ be the projection map from $Z\times\mathcal{R}^0\to Z$. It is obvious that $\mathrm{proj}_Z(G)$ is open if $G\subset Z\times \mathcal{R}^0$ is open. In addition, by the definition of $V$, for $c\in\R$, it holds
  \begin{align*}
    \big\{(s,x,i)\in Z;\ V(s,x,i)<c\big\}=\mathrm{proj}_Z\big\{(s,x,i,R)\in Z\times \mathcal{R}^0; J(s,x,i, R)<c\big\}.
  \end{align*}
  The upper semicontinuity of $f$ and $g$ yields the upper semicontinuity of the function $J$, and hence
  $\big\{(s,x,i,R)\in Z\times \mathcal{R}^0; J(s,x,i, R)<c\big\}$ is open, so
  $\big\{(s,x,i)\in Z;\ V(s,x,i)<c\big\}$ is open, which yields that $V$ is upper semicontinuous immediately.
\end{proof}

Applying \cite[Proposition 7.33]{BS}, a measurable selection theorem over general metric spaces, we can establish the following selection theorem in the current situation.
\begin{mylem}[measurable selection theorem]\label{lem-sel}
Assume that $f:[0,T]\times \R^d\times \S\times \pb(U)\times \pb(U)\to \R$, and $g:\R^d\to \R$ are lower semicontinuous, then there exists a measurable function $H:[0,T]\times\R^d\times \S\to \mathcal{R}^0$ such that
\[V(s,x,i)=J(s,x,i,H(s,x,i)).\]
\end{mylem}
\begin{proof}
  We can get the desired conclusion from \cite[Proposition 7.33]{BS} by taking the  spaces $X$ and $Y$ there as $[0,T]\!\times\!\R^d\!\times\!\S$ and $\mathcal{R}^0$ respectively, and applying Theorem \ref{t3}(a).
\end{proof}

To proceed, we adopt the method and notations of \cite{HS95b} and Stroock and Varadhan \cite{SV} to establish the dynamic programming principle.
According to \cite[Lemma 3.3, Corollary 3.9]{HS95b}, under the help of the selection theorem established in Lemma \ref{lem-sel}, the following result holds.

\begin{mylem}\label{lem-p}
For every $R\in \mathcal R_{s,x,i}$, $s<t\leq T$, there exists a unique probability measure on $\mathcal Y$, denoted by $R\otimes_t Q$, such that
\begin{itemize}
  \item[$(1)$] $R\otimes_t Q(A)=R(A)$, $\forall\, A\in \wt{\mathcal{Y}}_t$.
  \item[$(2)$] The regular conditional probability distribution of $R\otimes_t Q$ with respect to $\wt{\mathcal{Y}}_t$ is $Q_{t}$, where
      $Q_t=H(t,X_t,\La_t)$,
      $H(s,x,i)\in \mathcal{R}^0_{s,x,i}$, and $H:[0,T]\times \R^d\times \S\ra \pb(\mathcal Y)$ is Borel measurable.
  \item[$(3)$] $R\otimes_t Q$ is the distribution of the process $(X_\cdot,\La_\cdot,\mu_\cdot,\nu_\cdot)$ associated with some $\alpha=(\mu_\cdot,\nu_\cdot)\in \Pi_{s,x,i}$.
\end{itemize}
\end{mylem}

\begin{mythm}\label{t4}
Assume that (H1)-(H4) hold, and $f:[0,T]\times \R^d\times \S\times \pb(U)\times \pb(U)\to \R$, and $g:\R^d\to \R$ are lower semicontinuous. Then for $0\leq s<t\leq T$,
\begin{equation}\label{k-3}
V(s,x,i)=\inf\Big\{ \E\Big[\int_s^t f(r,X_r,\La_r,\mu_r,\nu_r)\d r+V(t,X_t,\La_t)\Big];\ \alpha\in  \Pi_{s,x,i}\Big\}.
\end{equation}
\end{mythm}

\begin{proof}
Let $\alpha\in \Pi_{s,x,i}$ and denote by $R$ the distribution of $(X_\cdot,\La_\cdot,\mu_\cdot,\nu_\cdot)$ in $\mathcal Y$ associated with $\alpha$.
By Lemma \ref{lem-p}, there exists an $\tilde\alpha=(\tilde \mu_\cdot,\tilde \nu_\cdot,s,x,i)\in \Pi_{s,x,i}$  associated with $R\otimes_t Q$. Then,
\begin{align*}
  &V(s,x,i)\\
  &\leq \E  \Big[\int_s^T f(r,\wt X_r,\tilde \La_r,\tilde\mu_r,\tilde\nu_r) \d r+g(\wt X_T)\Big]\\
  &=\E \Big[\int_s^t f(r,\wt X_r,\tilde \La_r,\tilde\mu_r,\tilde\nu_r) \d r+\int_t^T f(r,\wt X_r,\tilde \La_r,\tilde\mu_r,\tilde\nu_r)  \d r+g(\wt X_T)\Big]\\
  &=\E \Big[\int_s^t \!\!f(r,X_r, \La_r,\mu_r,\nu_r) \d r+\E \Big[\int_t^T \!\!\!f(r,\wt X_r,\tilde \La_r,\tilde\mu_r,\tilde\nu_r) \d r+g(\wt X_T)\Big|\F_{t}\Big]\Big]\\
  &=\E \Big[\int_s^t f(r, X_r, \La_r,\mu_r,\nu_r) \d r+V(\tau,X_\tau,\La_\tau)\Big].
\end{align*}
In the second equality of the previous equation, we have used that before $t$, $\tilde \alpha$ coincides with $\alpha$, and after $t$,  coincides with the measurable selector $H(t,X_t,\La_t)$.
The arbitrariness of $\alpha\in\Pi_{s,x,i}$ yields that
\begin{equation}\label{k-4}
V(s,x,i)
\leq \inf\Big\{ \E\Big[\int_s^t f(r,X_r,\La_r,\mu_r,\nu_r)\d r+V(t,X_t,\La_t)\Big];\ \alpha\in \Pi_{s,x,i}\Big\}.
\end{equation}

 On the other hand, by Theorem \ref{t1}, there exists an optimal admissible control $\alpha^\ast=(\mu_\cdot^\ast,\nu_\cdot^\ast,s,x,i)\in \Pi_{s,x,i}$. Denote by $(X_\cdot^\ast,\La_\cdot^\ast)$ the processes associated with $\alpha^\ast$. Then,
 \begin{equation}\label{k-5}
 \begin{split}
V(s,x,i)&
   =\E\Big[\int_s^T\!f(t,X_t^\ast,\La_t^\ast,\mu_t^\ast,\nu_t^\ast)\d t+g(X_T^\ast)\Big]\\
   &=\E\Big[\int_s^t\!f(r,X_r^\ast,\La_r^\ast,\mu_r^\ast,\nu_r^\ast)\d r+\int_t^T\!\!f(r,X_r^\ast,\La_r^\ast,\mu_r^\ast,\nu_r^\ast)\d r +g(X_T^\ast)\Big]\\
   &\geq \E\Big[\int_s^t\!f(r,X_r^\ast,\La_r^\ast,\mu_r^\ast,\nu_r^\ast)\d r+V(t,X_t^\ast,\La_t^\ast)\Big]\\
   &\geq \inf\Big\{ \E\Big[\int_s^t f(r,X_r,\La_r,\mu_r,\nu_r)\d r+V(t,X_t,\La_t)\Big];\ \alpha\in  \Pi_{s,x,i}\Big\}.
 \end{split}
 \end{equation}
 Consequently,   the dynamic programming principle \eqref{k-3} has been established following from \eqref{k-4} and \eqref{k-5}.
\end{proof}

\section*{Appendix}

\begin{mylem}\label{app-2}
  Let $Z$ be a set, $(E,\mathscr{E})$ a measurable space, and $\xi:Z\to E$ a map. Let $\sigma(\xi)=\xi^{-1}(\mathscr{E})$. Then $\varphi:Z\to \R$ is a $\sigma(\xi)$-measurable function if and only if there exists a measurable function $h:E\to \R$ such that $\varphi=h\circ \xi$.
\end{mylem}

\begin{proof}
  The sufficiency is clear, we shall use the functional monotone class theorem to prove the necessity. To this end, let
  \[L=\{h\circ \xi;\ h\in \mathscr{E}\}.\]
  Here $h\in \mathscr{E}$ means that $h$ is measurable w.r.t. $\mathscr{E}$.
  Then $1=\mathbf{1}_E\circ \xi$, $\mathbf{1}_{E}\in \mathscr{E}$, and hence $1\in L$. It is easy to check that $L$ is closed for linear combination. If $\varphi_n\in L$, $0\leq \varphi_n\uparrow \varphi$, then there exist $h_n\in \mathscr{E}$ such that
  $\varphi_n=h_n\circ \xi$. Let $h=\sup_n h_n$. Then $h\in \mathscr{E}$ and
  $\varphi=h\circ \xi$. So $\varphi\in L$.
  At last, for every $C\in \sigma(\xi)$, there exists a $B\in \mathscr{E}$ such that $C=\xi^{-1}(B)$. So, $\mathbf{1}_C=\mathbf{1}_B\circ\xi$. Consequently, according to the functional monotone class theorem, $L$ contains all the $\sigma(\xi)$-measurable functions.
\end{proof}

\noindent\textbf{Acknowledgements.} The author is grateful to the editor and referees for their valuable suggestions on the first version of this paper, which improve the quality of this paper.

\end{document}